\documentclass[sn-mathphys,Numbered]{sn-jnl}

\usepackage{graphicx}%
\usepackage{multirow}%
\usepackage{amsmath,amssymb,amsfonts}%
\usepackage{amsthm}%
\usepackage{mathrsfs}%
\usepackage[title]{appendix}%
\usepackage{xcolor}%
\usepackage{textcomp}%
\usepackage{manyfoot}%

\theoremstyle{thmstyleone}%
\newtheorem{theorem}{Theorem}
\theoremstyle{thmstyletwo}%
\newtheorem{example}{Example}%
\newtheorem{remark}{Remark}%

\theoremstyle{thmstylethree}%
\newtheorem{definition}{Definition}%

\theoremstyle{thmstyletwo}
\newtheorem{lemma}{Lemma}

\theoremstyle{thmstyletwo}
\newtheorem{corollary}{Corollary}

\begin{document}

\title[Approximation results on $s$-numbers of operators]{Approximation results on $s$-numbers of operators}


\author[]{\fnm{Deepesh} \sur{K. P.}}\email{deepeshkp@nitc.ac.in}



\affil[]{\orgdiv{Department of Mathematics}, \orgname{National Institute of Technology Calicut}, \orgaddress{\street{NIT Campus}, \city{Kozhikode}, \postcode{673601}, \state{Kerala}, \country{India}}}




\abstract{This article investigates the convergence properties of $s$-numbers of certain truncations of bounded linear operators between Banach spaces. We prove a generalized version of a known convergence result for the approximation numbers of truncations of operators, removing the restrictive assumption of separability from the underlying spaces. This generalization extends several existing results in the literature and establishes a close connection between two significant problems concerning approximation numbers. By exploring the relationships between approximation numbers and other prominent $s$-numbers, we also derive results on the convergence of $s$-numbers of truncations to those of the original operator, as applications of the generalized convergence result.
}




\maketitle

\section{Introduction}\label{sec1}
Singular values and their extensions, commonly known as $s$-numbers, play a crucial role in both linear algebra and functional analysis. They are vital components of the theory of the \textit{geometry of Banach spaces}, as many operator ideals are constructed based on the summability properties of these numbers \cite{carl, bakery, hammou, pie}. Extensive research has been conducted to explore the relationships between various $s$-numbers and eigenvalues, leading to well-established results \cite{carl-weyl, carl-hinrich}. Beyond their theoretical significance, $s$-numbers have also found practical applications in fields such as image processing \cite{gregor}, numerical methods \cite{mathe}, and mathematical physics \cite{brownian}.

Among $s$-numbers, the concepts of approximation numbers, Hilbert numbers, Chang numbers, Weyl numbers, Kolmogorov numbers and Gelfand numbers are well known in the literature \cite{carl-weyl, carl, pie}. Their properties, estimates, extensions and importance in various fields form active areas of research \cite{dmk1, edmunds,hinrichs-carl,hinrichs,oikh,pie-4, prochno}. Of these $s$-numbers, the concept of approximation numbers stands out as the most extensive and well studied one. For a bounded linear operator $T$ from a normed linear space $X$ to a normed linear space $Y$, the $k^{\mbox{th}}$ approximation number is denoted by $a_k(T)$, and is defined for every natural number $k$ as 
$$a_k(T)=\inf\{\|T-F\|: F \in \mathcal{F}_k(X,Y)\},$$ 
where $\mathcal{F}_k(X,Y)$ represents the collection of all bounded linear, finite rank operators of rank less than $k$ from $X$ to $Y$. The pioneering works of Albrecht Pietsch on the axiomatic study of $s$-numbers have played a historic role in advancing this field \cite{pie,pie-2}.
 
For a quantity defined for bounded linear operators between infinite dimensional spaces, it is a standard mathematical quest to investigate the possibility of approximating this quantity using the corresponding values of some truncations (a sequence of operators that converges to the original operator in some suitable sense of convergence) of the operator.  In this article, we refer to this problem as ``the convergence problem" and we assume that the sense of convergence of truncations is weaker than the norm convergence.

This convergence problem was addressed in \cite{bcn} for approximation numbers of bounded linear operators between separable Hilbert spaces. In that work, the authors used standard truncations by means of orthogonal projections, which converge to the identity operator in the pointwise sense, to establish positive results. Subsequently, in \cite{dmk}, the authors established a convergence result for approximation numbers of bounded linear operators between normed linear spaces by employing appropriate truncations. However, it is important to note that certain restrictions were imposed in this work on the spaces involved, particularly the separability assumptions on the domain space and the predual of the codomain space. These constraints, although necessary for obtaining the results in \cite{dmk}, somewhat limited the applicability of the results obtained therein.

In this article, by using the concept of nets and their role in characterizing the weak* compactness of sets, we eliminate the separability assumptions previously imposed on both spaces mentioned above and achieve convergence results similar to those in \cite{dmk} under more general assumptions. The convergence result established for approximation numbers in this article leads to generalized versions of many results from \cite{dmk}. Notably, it establishes a correspondence between the convergence problem and the complete symmetry problem, which concerns the equality of approximation numbers of an operator and those of its adjoint operator (see \cite{hut}), in a broader context. 

As another application of the main convergence result, we consider other well-known $s$-numbers and prove certain convergence results for them. The convergence result obtained for Gelfand numbers highlights the significance of the generalized version of the convergence result obtained in this article, when compared to the corresponding result proved in \cite{dmk}.

In the next section, we introduce the fundamental notations, definitions, and known results that play a pivotal role in this article. In Section 3, we derive a generalized convergence result for approximation numbers and provide several consequent results, including the interrelation between the convergence problem and the complete symmetry problem concerning the approximation numbers of operators. Also we apply the generalized convergence result to obtain certain convergence results for other prominent $s$-numbers. We provide a simple example to illustrate the convergence result proved for Kolmogorov numbers, along with some potential directions for future work, in this section.

\section{Notations and Priliminaries}
We use the notations $X$, $Y$ and $Z$ to represent infinite dimensional normed linear spaces over $\mathbb{C}$, the set of all complex numbers. The notation $\mathbb{N}$ stands for the set of all natural numbers, and the closed unit ball of $X$ is referred to as $U_X$. The collection of all bounded linear operators from $X$ to $Y$ is denoted as $B(X, Y)$, and the abbreviations $B(X)$ and $X^{\prime}$ are used for $B(X, X)$ and $B(X, \mathbb{C})$, respectively. By a dual space, we mean a Banach space that is isometrically isomorphic to $Z^{\prime}$ for some normed linear space $Z$.
For every $k \in \mathbb{N}$, $\mathcal{F}_k(X, Y)$ denotes the set of all finite-rank operators in $B(X, Y)$  whose ranks are less than $k$. For an operator $T \in B(X, Y)$, the adjoint operator $T^{\prime}$ is defined as the operator in $B(Y^{\prime}, X^{\prime})$ given by
$$(T^{\prime}f)(x)=f(Tx), \,\,x\in X, \,\,f\in Y^{\prime}.$$ 
In addition to approximation numbers, various $s$-numbers are defined for operators in $B(X, Y)$ as generalizations of singular values. For a given operator $T \in B(X, Y)$ and for each $k \in \mathbb{N}$, the $k^{\mbox{th}}$ $s$-number (see \cite{pie-2}) is denoted as $s_k(T)$. We define some of the well-known $s$-numbers in the following (for more details, refer to \cite{carl, pie-2}).
\begin{itemize}
\item The $k^{\mbox{th}}$ Chang number of $T$ is
\begin{eqnarray*}
y_k(T)=\sup\left\{a_k(ST): S\in B(Y,\ell_2), \, \|S\|\leq 1\right\}.
\end{eqnarray*}
\item The $k^{\mbox{th}}$ Weyl number of $T$ is
\begin{eqnarray*}
x_k(T)=\sup\left\{a_k(TR): R \in B(\ell_2, X), \, \|R\|\leq 1\right\}.
\end{eqnarray*}
\item The $k^{\mbox{th}}$ Kolmogorov number of $T$ is
\begin{eqnarray*}
d_k(T)=\inf\left\{\|Q_N^Y T\|:\, N \mbox{ closed subspace of }Y,\, dim(N)<k\right\},
\end{eqnarray*}
where $Q_N^Y$ denotes the canonical surjection from the Banach space $Y$ on to the quotient space $Y/N$.
\item The $k^{\mbox{th}}$ Gelfand number of $T$ is
\begin{eqnarray*}
c_k(T)=\inf\left\{\|TJ_M^X\|:\, M \mbox{ subspace of }X,\, codim(M)<k\right\},
\end{eqnarray*}
where $\displaystyle J_M^X$ denotes the canonical injection from $M$ into the Banach space $X$.
\end{itemize}

We use a few concepts and results from topology.
A net in a set $S$ is any function defined from a directed set $(\mathcal{D}, \preceq)$ to $S$. A net's convergence is defined analogously to the convergence of a sequence. The following definition of a subnet is used in this article.

\begin{definition}\cite{joshy}\label{subnet}
Let $f:(\mathcal{D},\preceq)\to S$ be a net in $S$ and $g$ be a function from a directed set $(\mathcal{D}_1,\preceq_1)$ to $(\mathcal{D},\preceq)$ satisfying the conditions 
\begin{itemize}
\item[i.] for every $a \preceq_1 b$ in $\mathcal{D}_1$, $g(a) \preceq g(b)$ in $\mathcal{D}$.
\item[ii.] given $a\in \mathcal{D}$, there is a $\gamma\in \mathcal{D}_1$ such that for $b\in \mathcal{D}_1$ satisfying $\gamma \preceq b$ in $\mathcal{D}_1$, $a \preceq g(b)$.
\end{itemize}
Then $f \circ \,g : (\mathcal{D}_1, \preceq_1) \to S$ is called a subnet of $f$.
\end{definition}

 Note that the subnet of a convergent net is also convergent. The compactness of sets can be characterized by the convergence of subnets of nets.

\begin{theorem}\cite{joshy}\label{compact-net}
A subset $S$ of a topological space is compact if and only if every net in $S$ has a subnet which converges  in $S$.
\end{theorem}

 The Banach-Alaoglu theorem is a well-known result in functional analysis.

\begin{theorem}(Banach-Alaoglu Theorem, \cite{limaye}) \label{Banach-Alauglu}
The closed unit ball of a dual space $Z^{\prime}$ is compact in the $weak^*$ topology. 
\end{theorem}
 It is known that the weak$^*$ topology on a dual space $Z^{\prime}$ is not metrizable. However, the Banach-Alaoglu theorem, together with Theorem \ref{compact-net}, ensures the existence of weak$^*$ convergent subnets for every net in a closed ball of finite radius within a dual space.

\begin{remark}\label{wot}{\rm
Suppose $f:\mathbb{N}\to S$ is a sequence defined by $f(n)=a_n$ for $n\in \mathbb{N}$. Then the notation $(a_n)$ represents the sequence $f$. Similarly, a net $f:(\mathcal{D},\preceq) \to S$ such that $f(\alpha)=x_{\alpha}$ for $\alpha \in \mathcal{D}$ is denoted by $(x_{\alpha})$ in this article. Consequently, as a result of Theorem \ref{Banach-Alauglu}, for each net $f=(a_{\alpha})$ in $U_{Z^{\prime}}$, there exists a subnet $f\,\circ \,g:\mathcal{D}_1\to U_{Z^{\prime}}$ (where $\mathcal{D}_1$ is a directed set and $g :\mathcal{D}_1 \to \mathcal{D}$ is as defined in Definition \ref{subnet}), denoted by $(a_{\phi(\alpha)})$, which converges to some element $a\in U_{Z^{\prime}}$ in the sense that for each $\epsilon > 0$, there exists a $\beta \in \mathcal{D}$ such that for all $\alpha \in \mathcal{D}$ with $\beta \preceq \alpha$, $|a_{\phi(\alpha)}(z) - a(z)|<\epsilon$ for each $z\in Z$.}
\end{remark}

 A net of operators $(T_{\alpha})\subset B(X,Y)$ is said to converge to $T\in B(X,Y)$ in the 
\begin{itemize}
\item[(1)] \textbf{Norm topology} if $\|T_{\alpha}-T\|\to 0$.
\item[(2)] \textbf{Strong (or pointwise) topology} if for each $x\in X$, $\|T_{\alpha}x-Tx\|\to 0$.
\item[(3)] \textbf{Weak topology} if for each $f\in Y^{\prime}\mbox{ and } x\in X$, $|f(T_{\alpha}x)-f(Tx)|\to 0$.
\item[(4)] \textbf{Weak$^*$ topology} if $Y=Z^{\prime}$ and for each $x\in X,\, z\in Z$, $|T_{\alpha}x(z)-Tx(z)|\to 0$.
\end{itemize}
Note that the convergence of a net of operators here is as mentioned in Remark \ref{wot}.

\section{Main results}\label{main-theorem-6}
One of the interesting convergence problems concerning $s$-numbers focuses on the convergence of $s$-numbers of truncations of an operator between infinite dimensional spaces. Specifically, given an operator $T\in B(X, Y)$ and a natural number $k\in \mathbb{N}$, we consider a sequence of truncations $(T_n)$ of $T$ (where $T_n \to T$ as $n\to \infty$ in some suitable sense of operator convergence) and investigate whether $s_k(T_n)$ converges to $s_k(T)$ as $n\to \infty$. This problem has a direct positive answer when $T_n$ converges to $T$ in the norm due to the inequality 
\begin{align*}
|s_k(T_n)-s_k(T)|\leq \|T_n-T\|,
\end{align*}
which holds for all $s$-numbers (11.1.1, \cite{pie}). Therefore, the problem is worth investigating when $T_n$ converges to $T$ in a weaker form of convergence. 

As mentioned in the introduction, this problem was initially addressed in \cite{bcn} for the approximation numbers of an operator $T$ between infinite-dimensional separable Hilbert spaces. In their work, the authors examined truncations of the form $Q_nTP_n$, where  $P_n$ and $Q_n$ were assumed to be standard orthogonal projections that converge pointwise to the identity operator.
Subsequently, in \cite{dmk}, this problem was considered for approximation numbers of bounded linear operators between infinite-dimensional normed linear spaces and their specific truncations. The authors obtained convergence results under certain assumptions on the spaces involved, the truncations used, and the sense of operator convergence. In Theorem \ref{dmk-thm} below, we state the most general version of the convergence result proved in \cite{dmk}, in this regard.

\begin{theorem}\label{dmk-thm}(Theorem 3.3, \cite{dmk})
Let $X$ be a separable space, $Y$  be isometrically isomorphic to the dual space of a separable space $Z$, and $T\in B(X,Y)$. Let $(P_n)$ and $(Q_n)$ be sequences of operators in $B(X)$ and $B(Y)$ respectively such that for all $n\in\mathbb{N}$, $\|Q_n\|\|P_n\|\leq 1$ and $T_n := Q_nTP_n \to T$ as $n\to \infty$ in the weak* operator topology. Then for each $k\in\mathbb{N}$, 
$$\displaystyle\lim_{n\to \infty}a_k(T_n)=a_k(T).$$
\end{theorem}

A problem posed in \cite{dmk} was on the necessity of the assumptions made in its Theorem 3.3. In this article, we show that the separability assumptions made in Theorem \ref{dmk-thm} are quite redundant. We achieve this result by using certain basic topological concepts. It is worth noting that the main ingredient used in the proof of Theorem 3.3 of  \cite{dmk} was its Lemma 3.1. To generalize the result mentioned in Theorem \ref{dmk-thm}, we first establish a generalized version of Lemma 3.1 of \cite{dmk}.
\begin{lemma}\label{generalized-lemma} Let $k\in \mathbb{N}$ be fixed. Suppose $(T_n)$ is a uniformly bounded sequence in $\mathcal{F}_{k+1}(X,Y)$, where $Y=Z^{\prime}$ for some normed linear space $Z$. Then there  exists a $T\in \mathcal{F}_{k+1}(X,Y)$ and a subnet $(T_{h(\alpha)})_{\alpha \in \mathcal{D}}$ (where $\mathcal{D}$ is some directed set) of $(T_n)$ such that $(T_{h(\alpha)})$ converges to $T$ in the weak* topology on $B(X,Y)$. That is, for every $\epsilon>0$, there exists a $\gamma \in \mathcal{D}$ satisfying 
\begin{eqnarray*}
|(T_{h(\alpha)}x)(z)-(Tx)(z)|<\epsilon \quad \forall\, \alpha \in \mathcal{D} \,\, \mbox{ with } \gamma \preceq \alpha.
\end{eqnarray*}

\end{lemma}

\begin{proof} 
Since for each $n\in \mathbb{N}$, rank$(T_n):=k_n \leq k$, we can express 
\begin{eqnarray*}
T_{n}x=\displaystyle\sum_{i=1}^{k_n}\psi_{i}^{n}(x)w_{i}^{n} \quad \mbox{ for } x\in X, \,{\psi_i}^{n}\in X^{\prime},\, {w_i}^{n} \in Y.
\end{eqnarray*}

In the expression of $T_n$, without loss of generality, one may assume that the summation runs from $1$ to $k$. Due to the Auerbach's lemma (B.4.8, \cite{pie}, Proposition 2.3, \cite{dmk}), one can also assume that for each $n\in \mathbb{N}$, ${w_i}^{n}$ are in the unit ball $U_{Y}$ and ${\psi_i}^{n}$ are in $W_{X^{\prime}}$, a closed bounded ball of radius, say $P$, in $X^{\prime}$. Since $X^{\prime}$ and $Y$ are dual spaces of normed linear spaces, due to Banach-Alaoglu theorem, both these sets are weak* compact and hence is their finite product. So the net $f:\mathbb{N}\to \displaystyle\prod_{i=1}^{k} U_{Z^{\prime}}$ defined by $f(n)=({w_1}^{n}, {w_2}^{n}, \ldots, {w_k}^{n})$ has a subnet, say $f\circ \phi: \mathcal{D}_1 \to \displaystyle\prod_{i=1}^{k} U_{Z^{\prime}}$ denoted by
\begin{eqnarray*}
(w_1^{\phi(\alpha)},w_2^{\phi(\alpha)},\ldots,w_k^{\phi(\alpha)}),
\end{eqnarray*} 
which converges to some element $(w_1,w_2,\ldots,w_k)$ in $\displaystyle\prod_{i=1}^{k} U_{Z^{\prime}}$ in the weak* sense of convergence. Here $\mathcal{D}_1$ represents a directed set and $\phi$ is a function from $\mathcal{D}_1$ to $\mathbb{N}$ such that the conditions of a subnet (Definition \ref{subnet}) are satisfied. This convergence means that for any $\epsilon >0$, there exists some $\gamma_1\in \mathcal{D}_1$ such that for each $z\in Z$, $1\leq i \leq k$ and for all $\alpha \in \mathcal{D}_1$ with $\gamma_1 \preceq \alpha$,
\begin{eqnarray*}
|{w_i}^{\phi(\alpha)}(z)-{w_i}(z)|< \frac{\epsilon}{2kP}.
\end{eqnarray*}

Now consider the subnet $g\circ \phi$ of the sequence $g(n)=(\psi_1^{n},\psi_2^{n},\ldots,\psi_k^{n})$ given by $g\circ \phi(\alpha)=( {\psi_1}^{\phi(\alpha)}, {\psi_2}^{\phi(\alpha)},\ldots,  {\psi_k}^{\phi(\alpha)})$. Due to the weak* compactness of $\displaystyle\prod_{i=1}^{k} W_{X^{\prime}}$, this net also has a subnet $g\circ \phi\circ\eta: \mathcal{D}_2\to \displaystyle\prod_{i=1}^{k} W_{X^{\prime}}$ which converges in the weak* sense to some element, say, $(\psi_1,\psi_2,\ldots,\psi_k)$ in $\displaystyle\prod_{i=1}^{k} W_{X^{\prime}}$. Here $\mathcal{D}_2$ represents a directed set and the convergence means that there exists a $\gamma_2$ in $\mathcal{D}_2$  such that for each $x\in X, 1\leq i\leq k$ and for all $\alpha \in \mathcal{D}_2$ with $\gamma_2\preceq \alpha$,

\begin{eqnarray*}
|{\psi_i}^{\phi\circ \eta(\alpha)}(x)-\psi_i(x)|<\frac{\epsilon}{2k}.
\end{eqnarray*} 
Being a subnet of $f\circ \phi$, the net $f\circ\phi\circ \eta$ also converges to $(w_1,w_2,\ldots,w_k)$ in the weak$^*$ sense. That is, there exists some $\gamma \in \mathcal{D}_2$ such that $\gamma_2 \preceq \gamma, \,\gamma_1 \preceq \eta(\gamma) $ and for each $z\in Z$, $1\leq i \leq k$, and for all $\alpha \in \mathcal{D}_2$ with $\gamma \preceq \alpha$,
\begin{eqnarray*}
|{w_i}^{\phi\circ\eta (\alpha)}(z)-{w_i}(z)|< \frac{\epsilon}{2kP}.
\end{eqnarray*}

Using these $w_i, \, \psi_i,\,\, 1\leq i\leq k$, we define an operator $T$ as

\begin{equation*}
Tx=\displaystyle\sum_{i=1}^k \psi_i(x)w_i,\,\, x\in X.
\end{equation*} 
Clearly, $T: X\to Y$ is a bounded finite rank operator of rank at most $k$. We claim that there is a subnet of $(T_n)$, which converges to $T$ in the weak* operator topology on $B(X, Y)$. To see this, letting $h=\phi\circ\eta$ and $\mathcal{D}=\mathcal{D}_2$, for each $x\in X$ and $z\in Z$, we have
\begin{eqnarray*} |(T_{h(\alpha)}x)(z)-(Tx)(z)|\leq
\displaystyle\sum_{i=1}^{k}|\psi_i^{h(\alpha)}(x)w_i^{h(\alpha)}(z)-\psi_{i}(x)w_i(z)|.\end{eqnarray*}
Now for each $i\in \{1,2,\ldots,k\}$, $\|x\|\leq 1,\, \|z\|\leq 1$, we get

\begin{eqnarray*}
|\psi_i^{h(\alpha)}(x)w_i^{h(\alpha)}(z)-\psi_{i}(x)w_i(z)|&\leq&|\psi_i^{h(\alpha)}(x)w_i^{h(\alpha)}(z)-\psi_{i}^{h(\alpha)}(x)w_i(z)|\\
&&+|\psi_i^{h(\alpha)}(x)w_i(z)-\psi_{i}(x)w_i(z)|\\
&\leq&|\psi_i^{h(\alpha)}(x)||w_i^{h(\alpha)}(z)-w_i(z)|\\
&&+|w_i(z)||\psi_i^{h(\alpha)}(x)-\psi_i(x)|\\
&\leq&P|w_i^{h(\alpha)}(z)-w_i(z)|+|\psi_i^{h(\alpha)}(x)-\psi_i(x)|,
\end{eqnarray*}
which can be made less than $\displaystyle\frac{\epsilon}{k}$ for each $x\in U_X,
\,z\in U_Z$, whenever $\alpha \in \mathcal{D}$ with $\gamma \preceq \alpha$. This implies that for all such $\alpha\in \mathcal{D}$ and for each $x\in X, \, z\in Z$,

\begin{eqnarray*} |(T_{h(\alpha)}x)(z)-(Tx)(z)|<\epsilon,
\end{eqnarray*}
which proves our claim.
\end{proof}
Since this lemma generalizes the result in Lemma 3.1 of \cite{dmk} under relaxed assumptions, the corollaries of Lemma 3.1 given in \cite{dmk} also follow without much alterations. Here, we state a straightforward consequence of Lemma \ref{generalized-lemma}, whose proof is not provided here, as it closely resembles the proof of Corollary 3.2 in \cite{dmk}.

\begin{corollary}\label{attaining}
Let $Y$ be a dual space and $T\in B(X, Y)$. Then for each $k\in \mathbb{N}$, there exists $F\in \mathcal{F}_k(X,Y)$ such that $a_k(T)=\|T-F\|$.
\end{corollary}

Now, with the help of Lemma \ref{generalized-lemma}, we derive a generalized version of Theorem \ref{dmk-thm} below. The proof of Theorem \ref{main-convergence-theorem} blends ideas from \cite{deep} and \cite{dmk}. We include a proof of this crucial result here, for the sake of completeness. 

\begin{theorem}\label{main-convergence-theorem} Let $Y$ be a dual space of $Z$ and $T\in B(X, Y)$. Let $(P_n)$ and $(Q_n)$ be sequences of operators in $B(X)$ and $B(Y)$ respectively such that $\|P_n\|\|Q_n\|\leq 1$ for all $n\in \mathbb{N}$. If
$T_{n}:=Q_nTP_n$ converges to $T$ in the weak$^*$ topology on $B(X,Y)$, as $n\to \infty$, then for each $k \in \mathbb{N}$,
\begin{eqnarray*} 
\displaystyle\lim_{n\rightarrow \infty}
a_{k}(T_{n})=a_{k}(T).
\end{eqnarray*}

\end{theorem}

\begin{proof}

Fix $k\in \mathbb{N}$ and let $d:= a_{k}(T)$. For $n\in \mathbb{N}$, let $d_{n}:= a_{k}(T_{n})$ and consider $F \in \mathcal{F}_k(X,Y)$. Then
$$a_k(T_n)\leq \|T_n - Q_nFP_n\|\leq \|T-F\|.$$ Then we get $a_k(T_n)\leq a_k(T)$ for all $n\in\mathbb{N}$. Hence
\begin{eqnarray*}
\sup_{n}d_n=\displaystyle\sup_{n}a_k(T_n) \leq a_k(T)=d.
\end{eqnarray*} 
Since for $d=0$, the conclusion is trivial, assume that $d_n
\not\to d$ with $d>0$. Then there exists an $\epsilon
>0$ and infinitely many $n$ such that $d_n<d-\epsilon$. Hence there
exist operators $F_{n_j}\in \mathcal{F}_k(X,Y)$ such that for all $j\in \mathbb{N}$,
\begin{eqnarray*} \|Q_{n_j}TP_{n_j}-F_{n_j}\|<d-\epsilon.\end{eqnarray*}
Thus for all $j\in \mathbb{N}$,
\begin{eqnarray*} \|F_{n_j}\|\leq
\|F_{n_j}-Q_{n_j}TP_{n_j}\|+\|Q_{n_j}TP_{n_j}\|<d+\|T\|.\end{eqnarray*}
Hence by Lemma \ref{generalized-lemma}, there exist an $F\in
\mathcal{F}_k(X,Y)$ and a subnet $(F_{h(\alpha)})$ of $(F_{n_{j}})$, where $\alpha$ is from a directed set $\mathcal{D}$,
such that $F_{h(\alpha)}$ converges to $F$ in the weak* sense of convergence. Now, let $x\in X,\, z\in Z$ be such that $\|x\|\leq 1,~\|z\|\leq 1$. Then for each $\alpha \in \mathcal{D}$, we have
\begin{eqnarray*}
|(Tx)(z)-(Fx)(z)|&\leq &|
(Tx)(z)-(Q_{h(\alpha)}TP_{h(\alpha)}x)(z)|\\
&+&|(Q_{h(\alpha)} TP_{h(\alpha)}x)(z)-(F_{h(\alpha)}x)(z)|\\
&+&|(F_{h(\alpha)}x)(z)-(Fx)(z)|.
\end{eqnarray*}
Note that for each $\alpha\in D$,
\begin{eqnarray*} 
|(Q_{h(a)}TP_{h(\alpha)}x)(z)-(F_{h(\alpha)}x)(z)| \leq
\|Q_{h(\alpha)}TP_{h(\alpha)}-F_{h(\alpha)}\|<d-\epsilon,
\end{eqnarray*}
whereas both the terms $|(Tx)(z)-(Q_{h(\alpha)}TP_{h(\alpha)}x)(z)|$ and
$|(F_{h(\alpha)}x)(z)-(Fx)(z)|$ can be made less than $\displaystyle\frac{\epsilon}{3}$ by
choosing suitable $\alpha$ from $D$. Hence
\begin{eqnarray*} 
|(Tx)(z)-(Fx)(z)| < d- \frac{\epsilon}{ 3}.
\end{eqnarray*}
Since this holds for each $x\in X$ and $z\in Z$ with
$\|x\|\leq 1,\, \|z\|\leq 1$, we have
\begin{eqnarray*} \|T-F\|\leq d- \displaystyle\frac{\epsilon}{3}.\end{eqnarray*}
Thus $d \leq d- \displaystyle\frac{\epsilon}{3}$, which is a contradiction. Hence $d_n \rightarrow d$ as $n\rightarrow \infty$.
\end{proof}
It can be easily observed that the result in Theorem \ref{main-convergence-theorem} holds if $Y$ is isometrically isomorphic to a dual space $Z^{\prime}$.
Theorem \ref{main-convergence-theorem} allows us to establish the convergence of approximation numbers for truncations of the adjoint of an operator without requiring any additional assumptions on the spaces involved.

\begin{corollary}\label{coro-convergence-transpose} 
Let $X, Y$ be any normed linear spaces, $T\in B(X,Y)$ and for each $n\in \mathbb{N}$, let $P_n,\, Q_n$ be as in Theorem \ref{main-convergence-theorem}. Suppose that $T_n:=Q_nTP_n \to T$ as $n\to \infty$ in the weak operator topology on $B(X,Y)$. Then
\begin{eqnarray*} 
\displaystyle \lim_{n\rightarrow
\infty}a_k(T_n^{\prime})=a_k(T^{\prime}).\end{eqnarray*}

\end{corollary}

\begin{proof} 
Here $X^{\prime}$ is the codomain of the adjoint operator. Since $T_n \to T$ in the weak sense, we have  $T_n^{\prime} \to T^{\prime}$ in the weak$^*$ sense. Further, for each $n\in \mathbb{N}$, $\|P_n^{\prime}\|\,\|Q_n^{\prime}\|= \|P_n\|\,\|Q_n\| \leq 1$. Hence, the result follows from Theorem \ref{main-convergence-theorem}.
\end{proof}
The concepts of symmetry and complete symmetry of $s$-numbers are well studied in the literature, and they play vital roles in the classification of operator spaces \cite{hut}.
\begin{definition}(\cite{pie})
An $s$-number function $(s_k(T))$ is said to be 
\begin{itemize}
\item[i.] symmetric if $s_k(T^{\prime})\leq s_k(T)$
\item[ii.] completely symmetric if $s_k(T^{\prime})= s_k(T)$,
\end{itemize}
for all $T\in B(X,Y)$ and for each $k\in\mathbb{N}$. 
\end{definition}
Mathematicians have investigated the occurrence of equality in the definition of complete symmetry for particular classes of operators. Regarding approximation numbers, it was demonstrated in \cite{hut} that for any compact operator $T$, $a_k(T^{\prime}) = a_k(T)$ for all $k \in \mathbb{N}$.

In \cite{dmk}, using Theorem \ref{dmk-thm}, the authors demonstrated a close relationship between the convergence problem and the complete symmetry problem (for approximation numbers) under certain assumptions. Our main result in Theorem \ref{main-convergence-theorem} allows us to establish a strong connection between these two problems without imposing any separability assumptions on the spaces.

\begin{theorem}\label{equality-convergence-thm} 
Suppose that there exist operators $P_n$ and $Q_n$ in $B(X)$ and $B(Y)$ respectively such that for all $n\in \mathbb{N}$, $T_n:=Q_nTP_n$ are compact operators in $B(X,Y)$, $\|P_n\|\|Q_n\|\leq 1$, and $(T_n)$ converge to $T$ in the weak sense of operator convergence. Then for each $k\in \mathbb{N}$,
\begin{eqnarray*} 
 \lim_{n\rightarrow
\infty}a_k(T_n)=a_k(T) \iff a_k(T)=a_k(T^{\prime}).
\end{eqnarray*}
\end{theorem}

\begin{proof}
By Corollary \ref{coro-convergence-transpose}, $\displaystyle \lim_{n\rightarrow \infty}a_k(T_n^{\prime})=a_k(T^{\prime})$. If  $T_n$ is compact for each $n\in\mathbb{N}$, we have $a_k({T_n}^{\prime}) = a_k({T_n})$, by \cite{hut}. The result follows from the uniqueness of limits.
\end{proof}

\begin{remark}{\rm
The example given in Proposition 4.4 of \cite{dmk} shows the necessity of the assumptions made on the codomain space in Theorem \ref{main-convergence-theorem} of this article.}
\end{remark}

The convergence problem discussed for approximation numbers is also valid for other $s$-numbers. As mentioned at the begining of this section, due to the inequality
$|s_k(T_n)-s_k(T)|\leq \|T_n-T\|$, the convergence problem has a positive answer for all $s$-numbers, provided $(T_n)$ converges to $T$ in the norm sense. As another direct application of Theorem \ref{main-convergence-theorem} (Theorem \ref{main-theorem-6} also suffices), we observe a convergence result for $s$-number of operators between Hilbert spaces.
\begin{theorem}\label{hilbertspace-equal}
Let $H_1,H_2$ be Hilbert spaces and $T\in B(H_1,H_2)$. Suppose $(P_n)$ and $(Q_n)$ are sequences of operators in $B(H_1)$ and $B(H_2)$ respectively such that for all $n\in \mathbb{N}$, $\|P_n\|\|Q_n\|\leq 1$ and $T_n:=Q_nTP_n$ converges to $T$ in the weak* sense of operator convergence. Then for each $k\in \mathbb{N}$, $ \lim_{n\rightarrow
\infty}s_k(T_n)=s_k(T)$.
\end{theorem}
\begin{proof}
From Theorem 11.3.4 of \cite{pie}, all $s$-numbers coincide with the approximation numbers for operators between Hilbert spaces. Hence under the given assumptions, Theorem \ref{main-convergence-theorem} can be used to conclude the result.
\end{proof}

The inequality mentioned above and Theorem \ref{hilbertspace-equal} strongly suggest that the convergence problem may have a positive answer when considered for other $s$-numbers of operators between Banach spaces as well.
In the remaining part of this article, we assume that $X$ and $Y$ are Banach spaces, and that $P_n$ and $Q_n$ are operators in $B(X)$ and $B(Y)$ respectively, satisfying $\|Q_n\|\leq 1$ and $\|P_n\|\leq 1$ for each $n\in \mathbb{N}$. Here, we consider only one sided truncations of an operator $T\in B(X,Y)$ (that is, either $T_n:=Q_nT$ or $T_n:=TP_n$) that converge to $T$ in the weak* sense of operator convergence, in order to obtain convergence results for other $s$-numbers as applications of Theorem \ref{main-convergence-theorem}. 
We start with a convergence result for the Chang numbers.

\begin{theorem}\label{chang-numbers}
Let $T\in B(X,Y)$ and $T_n=TP_n$ converges to $T$ in the weak* sense as $n\to \infty$. Then for each $k\in \mathbb{N}$,
\begin{eqnarray*}
\displaystyle\lim_{n\to \infty} y_k(T_n)=y_k(T).
\end{eqnarray*}
\end{theorem}

\begin{proof}
For each $n\in \mathbb{N}$, since $a_k(STP_n)\leq a_k(ST)$, we have  
\begin{eqnarray*}
y_k(T_n)=\sup\left\{a_k(ST_n): S\in B(Y,\ell_2), \, \|S\|\leq 1\right\}\leq y_k(T).
\end{eqnarray*}
Here the codomain of $ST_n$ is reflexive and $\|P_n\|\leq 1$ for $n\in \mathbb{N}$. Hence for each such $S$,
\begin{eqnarray*}
a_k(ST_n)\to a_k(ST)
\end{eqnarray*}  
as $n\to \infty$, by Theorem \ref{main-convergence-theorem}. Now for every $\epsilon >0$, there exists an $\tilde{S}\in B(Y,\ell^2)$ such that $y_k(T) < a_k(\tilde{S}T)+\displaystyle\frac{\epsilon}{2}$. For this $\tilde{S}$ also, $y_k(T_n) \geq a_k(\tilde{S}T_n)$ for all $n\in \mathbb{N}$. So we have 
\begin{eqnarray*}
|y_k(T)-y_k(T_n)|< a_k(\tilde{S}T)+\frac{\epsilon}{2}-a_k(\tilde{S}T_n),
\end{eqnarray*}
which can be made less than $\epsilon$ for sufficiently large $n$, proving our claim.
\end{proof}

\begin{remark}{\rm
Note that in the above result, the codomain of $ST$ is a reflexive space and so we do not have to impose separability assumptions on the  spaces involved.}
\end{remark}

However, to obtain a similar result for Weyl numbers, we need to assume the separability of the codomain space of the operator, in order to apply Theorem \ref{main-convergence-theorem}. The proof technique is similar to that of the previous theorem. 

\begin{theorem}\label{weyl-numbers}
Let $Y$ be a dual space, $T\in B(X,Y)$ and $T_n:=Q_nT$ converges to $T$ as $n\to \infty$, in the weak* sense of convergence. Then for each $k\in \mathbb{N}$,
\begin{eqnarray*}
\displaystyle\lim_{n\to \infty} x_k(T_n)=x_k(T).
\end{eqnarray*}
\end{theorem}

\begin{proof}
For each $n\in \mathbb{N}$, we have 
\begin{eqnarray*}
x_k(T_n)=\sup\left\{a_k(T_n R): R\in B(\ell^2,X), \, \|R\|\leq 1\right\}.
\end{eqnarray*}
Since $Y$ is assumed to be a dual space, we can apply Theorem \ref{main-convergence-theorem} to $a_k(T_n R)$, and arguments similar to the proof of Theorem \ref{chang-numbers} leads to the conclusion.
\end{proof}

To prove similar convergence results for Kolmogorov and Gelfand numbers, we use certain general Banach spaces known in the literature \cite{carl}.

For an arbitrary index set $A$, the space $\ell^1(A)$ denotes the Banach space of all absolutely summable complex number families $(x_{\alpha})_{\alpha \in A}$. That is,
\begin{eqnarray*}
\ell^{1}(A)=\left\{(x_{\alpha}):\,\, \sum_{\alpha \in A}|x_{\alpha}|<\infty \right\} \mbox{ with }\|(x_{\alpha})\|_1=\sum_{\alpha \in A}|x_{\alpha}|,
\end{eqnarray*}
assuming $x_{\alpha}=0$ for all but countably many $\alpha \in A$ in the summation. The Banach space $\ell^{\infty}(A)$ denotes the collection of all bounded complex number families defined over the index set $A$. That is,
\begin{eqnarray*}
\ell^{\infty}(A)=\left\{(x_{\alpha}):\,\, \sup_{\alpha \in A}|x_{\alpha}|<\infty \right\} \mbox{ with }\|(x_{\alpha})\|_{\infty}=\sup_{\alpha \in A}|x_{\alpha}|.
\end{eqnarray*}
The space $c_{00}(A)$ with  $\|\cdot\|_{\infty}$ is defined analogous to the usual sequence space $c_{00}$.
\begin{eqnarray*}
c_{00}(A)=\left\{(x_{\alpha}): x_{\alpha}=0 \,\, \mbox{ for all } \alpha \notin \tilde{A}, \,\mbox{where }\tilde{A} \mbox{ is any finite subset of }A \right\}.
\end{eqnarray*}
Apart from being generalizations of the usual sequence spaces, these spaces play important roles in the geometry of Banach spaces. It is known (\cite{carl}) that every Banach space $X$ is isometrically isomorphic to the quotient space $\ell^1(U_X)/{M}_0$ with the quotient norm, where $M_0$ is the closed subspace of $\ell^1(U_X)$ given by
\begin{eqnarray*}
M_0 :=\left\{(x_{\alpha})_{\alpha \in U_X}:\,\, \sum_{\alpha \in U_X}\alpha x_{\alpha}=0 \right\}\subset \ell^1(U_X).
\end{eqnarray*}
For an operator whose codomain is a Banach space with the \textit{metric lifting property}, the approximation numbers coincide with the corresponding Kolmogorov numbers (see Proposition 2.2.3 in \cite{carl} for details). For operators between arbitrary Banach spaces, the following result establishes the relationship between these two types of numbers. 

\begin{theorem}\label{carl-kolmogorov}(Theorem 2.2.1, \cite{carl})
Let $Q_X$ be the quotient map from $\ell^1(U_X)$ onto $X$ and $T\in B(X,Y)$. Then for each $k\in \mathbb{N}$,
\begin{eqnarray*}
d_k(T)=a_k(TQ_X).
\end{eqnarray*}
\end{theorem}
Using this relation, we obtain the following convergence result for Kolmogorov numbers, as a consequence of Theorem \ref{main-convergence-theorem}.

\begin{theorem}\label{Kolmogorov-conv}
Let $Y$ be a dual space, $T\in B(X,Y)$, and let $T_n:=Q_nT$ converges to $T$ in the weak* sense of convergence. Then for each $k\in \mathbb{N}$,
$$\displaystyle\lim_{n\to \infty}d_k(T_n)=d_k(T).$$
\end{theorem}
\begin{proof}
Here $d_k(T)=a_k(TQ_X)$, where $Q_X$ is the quotient map from $\ell^1(U_X)$ on to the space $\ell^1(U_X)/M_0 \equiv X$ (page 52, \cite{carl}). Since $T_nQ_X$ converges to $TQ_X$ in the weak* sense and $Y$ is assumed to be a dual space, by applying Theorem \ref{main-convergence-theorem} to $a_k(Q_nTQ_X)$, we obtain the proof. 
\end{proof}
Regarding Gelfand numbers, it is known that if $Y$ is a Banach space with the \textit{metric extension property} and $X$ is any Banach space, then the Gelfand numbers of operators in $B(X, Y)$ coincide with the corresponding approximation numbers (see Proposition 2.3.3 in \cite{carl} for details). 

For an operator between two general Banach spaces, an expression for Gelfand numbers in terms of the approximation numbers is given in the following theorem. This result utilizes the fact that every Banach space $X$ can be viewed as a subspace of $\ell^{\infty}(U_{X^{\prime}})$, where $U_{X^{\prime}}$ denotes the closed unit ball of the dual space $X^{\prime}$.

\begin{theorem}\label{carl-gelfand}(Theorem 2.3.1, \cite{carl})
Let $J_Y$ be the embedding of $Y$ into $\ell^{\infty}(U_{Y^{\prime}})$ and $T\in B(X,Y)$. Then for each $k\in \mathbb{N}$,
\begin{eqnarray*}
c_k(T)=a_k(J_Y T).
\end{eqnarray*}
\end{theorem}
Now we show that the codomain of the operator $J_Y$ is indeed a dual space, allowing us to apply Theorem \ref{main-convergence-theorem} to $a_k(J_Y T_n)$. The proof is essentially a repetition of the standard functional analysis techniques used to establish the dual spaces of sequence spaces, specifically applied to $c_{00}(A)$. We include a proof here, as we were unable to find it in this form in the literature.

\begin{lemma}\label{duality of general sequence}
For an arbitrary index set $A$, the dual space of $(c_{00}(A),\|\cdot\|_{1})$ is isometrically isomorphic to $(\ell^{\infty}(A), \|\cdot \|_{\infty})$. 
\end{lemma}

\begin{proof}
For each $\alpha\in A$, consider the element $\chi_{\alpha}\in c_{00}(A)$ defined by

\begin{eqnarray*}
\chi_{\alpha}(\beta)=
\begin{cases}
1 \mbox{ if }  \beta = \alpha\\
0 \mbox{ if } \beta \neq \alpha,
\end{cases}
\end{eqnarray*}
for all $\beta \in A$. Then $\|\chi_{\alpha}\|_1=1$, and for each $f \in {c_{00}}^{\prime}(A)$, the family $(f(\chi_{\alpha})) \in \ell^{\infty}(A)$, since $|f(\chi_{\alpha})|\leq \|f\|$ for $\alpha \in A$. We define a linear map $T: {c_{00}}^{\prime}(A) \to \ell^{\infty}(A)$ defined by
\begin{eqnarray*}
T(f)=(f(\chi_{\alpha})),\quad f\in {c_{00}}^{\prime}(A),
\end{eqnarray*}
and observe that this is an isometrical isomorphism.
To see that $T$ is onto, for $(y_{\alpha})\in \ell^{\infty}(A)$, consider the map $f$ given by $$f((x_{\alpha}))=\displaystyle\sum_{\alpha \in A}x_{\alpha}y_{\alpha} \mbox{ for } (x_{\alpha})\in c_{00}(A).$$
Then $f\in {c_{00}}^{\prime}(A)$ and $Tf=(y_{\alpha})$. Now to see that $T$ is an isometry, observe that 
$$\|f\|=\displaystyle\sup\{|f((x_{\alpha}))|:\, (x_{\alpha})\in c_{00}(A),\, \|(x_{\alpha})\|_1=1\}\geq |f(\chi_{\beta})|$$
for each $\beta \in A$, showing that 
$\|f\|=\displaystyle\sup_{\alpha \in A}|f(\chi_{\alpha})|=\|(f(\chi_{\alpha}))\|_{\infty}=\|Tf\|_{\infty}$.
\end{proof}
By applying Theorem \ref{main-convergence-theorem} to the approximation numbers of $J_Y T$ mentioned in Theorem \ref{carl-gelfand}, we obtain a convergence result for Gelfand numbers of operators between Banach spaces. It is important to note that, in this case also, there is no need to impose separability restrictions on the codomain space of the operator.

\begin{theorem}\label{Gelfand-conv}
Let $T\in B(X,Y),\, T_n=TQ_n$ for $n\in N$ and $T_n \to T$ in the weak* sense of convergence. Then for each $k\in \mathbb{N}$,
\begin{eqnarray*}
\displaystyle\lim_{n\to \infty}c_k(T_n)=c_k(T).
\end{eqnarray*}
\end{theorem}
\begin{proof}
For each $n\in\mathbb{N}$, we have $c_k(T_n) = a_k(J_YT_n)$ by Theorem \ref{carl-gelfand}. Now the codomain of $J_Y T$ is $\ell^{\infty}(U_{Y^{\prime}})$, which is a dual space by Lemma \ref{duality of general sequence}. Hence by applying Theorem \ref{main-convergence-theorem} to $a_k(J_YT)$, we get the conclusion.
\end{proof}

\begin{remark}{\rm
Note that, in general, the space $\ell^{\infty}(U_{Y^{\prime}})$ is not the dual space of any separable space. This highlights the advantage of the result given in Theorem \ref{main-convergence-theorem} of this article over Theorem 3.3 of \cite{dmk}.}
\end{remark}  

We now provide an example to illustrate Theorem \ref{Kolmogorov-conv}.

\begin{example}{\rm
Consider the identity operator $I:\ell^1 \to \ell^1$. From Lemma 11.6.7 of \cite{pie}, we have $d_k(I)=1$ for all $k\in \mathbb{N}$.
Now consider the standard projections $Q_n:\ell^1 \to \ell^1$ given by 
$$Q_n(x_1, x_2, \ldots)=(x_1, x_2, \ldots, x_n, 0, 0, \ldots),\,\,(x_1, x_2, \ldots)\in \ell^1.$$ Then $Q_n \to I$ in the pointwise sense. Now from Theorem 11.11.3  of \cite{pie}, we get $$d_k(Q_n)=1\mbox{ for all }n\geq k.$$
Thus $d_k(T_n)\to d_k(T)$ as $n\to \infty$, where $T=I$ and $T_n=Q_nT$ for $n\in \mathbb{N}$.}
\end{example}
\section{Concluding remarks and potential research problems}
{Although the convergence problem has been addressed for approximation numbers, it is not known if the result in Theorem \ref{main-convergence-theorem} holds when the norm restrictions on $P_n$ and $Q_n$ are removed. This is a significant case to consider, as in Banach spaces, it is not always possible to obtain operators of norm $1$ or less that converge to the identity operator. Also, for the other $s$-numbers, only partial results were achieved by applying Theorem \ref{main-convergence-theorem}, as the truncations considered in the results here are not two sided. It would be of considerable interest to obtain similar convergence results for other $s$-numbers, either by modifying Theorem \ref{main-convergence-theorem} or through alternative approaches.} 
\vskip0.25cm

\noindent \textbf{ACKNOWLEDGEMENTS.}

I express my sincere gratitude to my PhD supervisors Prof. S. H. Kulkarni and Prof. M. T. Nair for their invaluable suggestions during the earlier phases of this research work. 

\end{document}